\documentclass[12pt,a4paper]{amsart}
\usepackage{amsfonts}
\usepackage{amsthm}
\usepackage{amsmath}
\usepackage{amscd}
\usepackage[latin2]{inputenc}
\usepackage{t1enc}
\usepackage[mathscr]{eucal}
\usepackage{indentfirst}
\usepackage{graphicx}
\usepackage{graphics}
\usepackage{pict2e}
\usepackage{epic}
\usepackage{color}
\numberwithin{equation}{section}
\usepackage[margin=2.9cm]{geometry}
\usepackage{epstopdf}

\newtheorem{theorem}{Theorem}[section]
\newtheorem{lemma}[theorem]{Lemma}
\newtheorem{remark}[theorem]{remark}
\newtheorem{proposition}[theorem]{Proposition}

\theoremstyle{definition}
\newtheorem{definition}[theorem]{Definition}
\newtheorem{example}[theorem]{Example}

\begin{document}
\setcounter{page}{1}
\title[Integral $K$-Operator Frames for $End_{\mathcal{A}}^{\ast}(\mathcal{H})$]{Integral $K$-Operator Frames for $End_{\mathcal{A}}^{\ast}(\mathcal{H})$}
\author[Hatim LABRIGUI and Samir KABBAJ]{Hatim LABRIGUI$^1$$^{*}$ and Samir KABBAJ$^1$ }
\address{$^{1}$Department of Mathematics, Ibn Tofail University, B.P. 133, Kenitra, Morocco}
\email{hlabrigui75@gmail.com; samkabbaj@yahoo.fr}
\date{
\newline \indent $^{*}$ Corresponding author}


\date{...}
\maketitle

\begin{abstract}
	In this work, we introduce a new concept of integral $K$-operator frame for the set of all adjointable operators from Hilbert $C^{\ast}$-modules $\mathcal{H}$ to  it self noted $End_{\mathcal{A}}^{\ast}(\mathcal{H}) $.  We give some propertis relating some construction of integral $K$-operator frame and operators preserving  integral $K$-operator frame and we establish some new results.
\end{abstract}

\section{Introduction and preliminaries}

In 1952 Duffin and Schaefer \cite{Duf} have introduced the concept of frames in the study of nonharmonic Fourier series. Frames possess many nice properties which make them very useful in wavelet analysis, irregular sampling theory, signal processing and many other fields. The theory of frames has been generalized rapidly and various generalizations of frames have emerged in Hilbert spaces and Hilbert $C^{\ast}$-modules (see \cite{F4, LG, mjpaa, moi,moi1, r4, r04, r5, r7, r9, r8}).

 The concept of continuous frames has been defined by Ali, Antoine and Gazeau \cite{STAJP}. Gabardo and Han in \cite{14} called these kinds frames or frames associated with measurable spaces. For more details, the reader can refer to \cite{ARAN}, \cite{MR1} and \cite{MR2}.\\

In this paper, we introduce a new concept of integral $K$-operator frame for the set of all adjointable operators from Hilbert $C^{\ast}$-modules $\mathcal{H}$ to  $\mathcal{H}$ noted $End_{\mathcal{A}}^{\ast}(\mathcal{H}) $, is a generalization of continuous K-frames for Hilbert $C^{\ast}$-modules, and we establish some new results.

In what follows, we set $\mathcal{H}$ a separable Hilbert space, $End_{\mathcal{A}}^{\ast}(\mathcal{H}) $ the set of all adjointable operators from Hilbert $C^{\ast}$-modules $\mathcal{H}$ to  $\mathcal{H}$ and $(\Omega,\mu)$ a measure space with positive measure $\mu$.\\
Let $K,T \in End_{\mathcal{A}}^{\ast}(\mathcal{H})  $, if $TK=I$, then $T$ is called the left inverse of $K$, denoted by $K_{l}^{-1}$.\\
If $KT=I$, then $T$ is called the right inverse of $K$ and we write $K_{r}^{-1}=T$.\\
If $KT=TK=I$, then $T$ and $K$ are inverse of each other.\\
For a separable Hilbert space $H$ and a measurable space $(\Omega,\mu)$, we define,
\begin{equation*}
l^{2}(\Omega,\mathcal{H})=\{x_{\omega} \in \mathcal{H},\quad  \omega \in \Omega,\quad \left\|\int_{\Omega}\langle x_{\omega},x_{\omega}\rangle d\mu(\omega)\right\| < \infty  \}.
\end{equation*}
For any $x=(x_{\omega})_{\omega \in \Omega}$ and $y=(y_{\omega})_{\omega \in \Omega}$, the inner product on $l^{2}(\Omega,H)$ is defined by, 
\begin{equation*}
\langle x,y\rangle = \int_{\Omega}\langle x_{\omega},y_{\omega}\rangle d\mu(\omega).
\end{equation*}
The norme is defined by $\|x\|=\langle x,x\rangle^{\frac{1}{2}}$.

In this section we briefly recall the definitions and basic properties of $C^{\ast}$-algebra, Hilbert $\mathcal{A}$-modules and frame in Hilbert $\mathcal{A}$-modules. For information about frames in Hilbert spaces we refer to \cite{Ch}. Our references for $C^{\ast}$-algebras are \cite{{Dav},{Con}}.\\
For a $C^{\ast}$-algebra $\mathcal{A}$, if $a\in\mathcal{A}$ is positive we write $a\geq 0$ and $\mathcal{A}^{+}$ denotes the set of positive elements of $\mathcal{A}$.
\begin{definition}\cite{Kap}.
	Let $ \mathcal{A} $ be a unital $C^{\ast}$-algebra and $\mathcal{H}$ be a left $ \mathcal{A} $-module, such that the linear structures of $\mathcal{A}$ and $ \mathcal{H} $ are compatible. $\mathcal{H}$ is a pre-Hilbert $\mathcal{A}$-module if $\mathcal{H}$ is equipped with an $\mathcal{A}$-valued inner product $\langle.,.\rangle :\mathcal{H}\times\mathcal{H}\rightarrow\mathcal{A}$, such that is sesquilinear, positive definite and respects the module action. In the other words,
	\begin{itemize}
		\item [(i)] $ \langle x,x\rangle_{\mathcal{A}}\geq0 $ for all $ x\in\mathcal{H} $ and $ \langle x,x\rangle_{\mathcal{A}}=0$ if and only if $x=0$.
		\item [(ii)] $\langle ax+y,z\rangle_{\mathcal{A}}=a\langle x,y\rangle_{\mathcal{A}}+\langle y,z\rangle_{\mathcal{A}}$ for all $a\in\mathcal{A}$ and $x,y,z\in\mathcal{H}$.
		\item[(iii)] $ \langle x,y\rangle_{\mathcal{A}}=\langle y,x\rangle_{\mathcal{A}}^{\ast} $ for all $x,y\in\mathcal{H}$.
	\end{itemize}	 
\end{definition}
For $x\in\mathcal{H}, $ we define $||x||=||\langle x,x\rangle||^{\frac{1}{2}}$. If $\mathcal{H}$ is complete with $||.||$, it is called a Hilbert $\mathcal{A}$-module or a Hilbert $C^{\ast}$-module over $\mathcal{A}$. For every $a$ in the  $C^{\ast}$-algebra $\mathcal{A}$, we have $|a|=(a^{\ast}a)^{\frac{1}{2}}$ and the $\mathcal{A}$-valued norm on $\mathcal{H}$ is defined by $|x|=\langle x, x\rangle^{\frac{1}{2}}$ for $x\in\mathcal{H}$.

Let $\mathcal{H}$ and $\mathcal{K}$ be two Hilbert $\mathcal{A}$-modules. A map $T:\mathcal{H}\rightarrow\mathcal{K}$ is said to be adjointable if there exists a map $T^{\ast}:\mathcal{K}\rightarrow\mathcal{H}$ such that $\langle Tx,y\rangle_{\mathcal{A}}=\langle x,T^{\ast}y\rangle_{\mathcal{A}}$ for all $x\in\mathcal{H}$ and $y\in\mathcal{K}$.

We also reserve the notation $End_{\mathcal{A}}^{\ast}(\mathcal{H},\mathcal{K})$ for the set of all adjointable operators from $\mathcal{H}$ to $\mathcal{K}$ and $End_{\mathcal{A}}^{\ast}(\mathcal{H},\mathcal{H})$ is abbreviated to $End_{\mathcal{A}}^{\ast}(\mathcal{H})$.
\begin{definition}\cite{BA}
	Let $ \mathcal{H} $ be a Hilbert $\mathcal{A}$-module over a unital $C^{\ast}$-algebra. A family $\{x_{i}\}_{i\in I}$ of elements of $\mathcal{H}$ is said to be a frame for $ \mathcal{H} $, if there
	exist two positive constants $A,B$ such that,
	\begin{equation}\label{eq1}
	A\langle x,x\rangle_{\mathcal{A}}\leq\sum_{i\in I}\langle x,x_{i}\rangle_{\mathcal{A}}\langle x_{i},x\rangle_{\mathcal{A}}\leq B\langle x,x\rangle_{\mathcal{A}}, \qquad x\in \mathcal{H}.
	\end{equation}
	The numbers $A$ and $B$ are called lower and upper bounds of the frame, respectively. If $A=B=\lambda$, the frame is called $\lambda$-tight. If $A = B = 1$, it is called a normalized tight frame or a Parseval frame. If the sum in the middle of \eqref{eq1} is convergent in norm, the frame is called standard. If only upper inequality of \eqref{eq1} holds, then $\{x_{i}\}_{i\in I}$ is called a Bessel sequence for $\mathcal{H}$.
\end{definition}
In \cite{LG}, L. Gavruta introduced $K$-frames to study atomic systems for operators in Hilbert spaces.
\begin{definition}\cite{Gav} Let $K\in End_{\mathcal{A}}^{\ast}(\mathcal{H})$. A family $\{x_{i}\}_{i\in I}$ of elements in a Hilbert $\mathcal{A}$-module $\mathcal{H}$ over a unital $C^{\ast}$-algebra is a $K$-frame for $ \mathcal{H} $, if there exist two positive constants $A$ and $B$, such that,
	\begin{equation}\label{11}
	A\langle K^{\ast}x,K^{\ast}x\rangle_{\mathcal{A}}\leq\sum_{i\in I}\langle x,x_{i}\rangle_{\mathcal{A}}\langle x_{i},x\rangle_{\mathcal{A}}\leq B\langle x,x\rangle_{\mathcal{A}}, \qquad x\in\mathcal{H}.
	\end{equation}
	The numbers $A$ and $B$ are called lower and upper bounds of the $K$-frame, respectively.
\end{definition}
The following lemmas will be used to prove our mains results
\begin{lemma} \label{l1} \cite{Pas}
	Let $\mathcal{H}$ be a Hilbert $\mathcal{A}$-module. If $T\in End_{\mathcal{A}}^{\ast}(\mathcal{H})$, then $$\langle Tx,Tx\rangle_{\mathcal{A}}\leq\|T\|^{2}\langle x,x\rangle_{\mathcal{A}},\qquad  \forall x\in\mathcal{H}.$$
\end{lemma}
\begin{lemma} \label{l2} \cite{Zha}
	Let $\mathcal{H}$ be a Hilbert $\mathcal{A}$-module over a $C^{\ast}$-algebra $\mathcal{A}$ and let $T, S$ two operators for $End_{\mathcal{A}}^{\ast}(\mathcal{H})$. If $Rang(S)$ is closed, then the following statements are equivalent:
	\begin{itemize}
		\item [(i)] $Rang(T)\subseteq Rang(S)$.
		\item [(ii)] $ TT^{\ast}\leq \lambda SS^{\ast}$ for some $\lambda>0$.
		\item [(iii)] There exists $Q\in End_{\mathcal{A}}^{\ast}(\mathcal{H})$ such that $T=SQ$.
	\end{itemize}
\end{lemma}
\begin{lemma} \label{l3} \cite{Ali}.
	Let $\mathcal{H}$ and $\mathcal{K}$ be two Hilbert $\mathcal{A}$-modules and $T\in End^{\ast}(\mathcal{H},\mathcal{K})$.
	\begin{itemize}
		\item [(i)] If $T$ is injective and $T$ has closed range, then the adjointable map $T^{\ast}T$ is invertible and $$\|(T^{\ast}T)^{-1}\|^{-1}\leq T^{\ast}T\leq\|T\|^{2}.$$
		\item  [(ii)]	If $T$ is surjective, then the adjointable map $TT^{\ast}$ is invertible and $$\|(TT^{\ast})^{-1}\|^{-1}\leq TT^{\ast}\leq\|T\|^{2}.$$
	\end{itemize}	
\end{lemma}
\begin{lemma} \label{l4} \cite{33}.
	Let $(\Omega,\mu )$ be a measure space, $X$ and $Y$ are two Banach spaces, $\lambda : X\longrightarrow Y$ be a bounded linear operator and $f : \Omega\longrightarrow X$ measurable function; then, 
	\begin{equation*}
	\lambda (\int_{\Omega}fd\mu)=\int_{\Omega}(\lambda f)d\mu.
	\end{equation*}
\end{lemma}
\section{Integral K-Operator Frames for $End_{\mathcal{A}}^{\ast}(\mathcal{H})$}
\quad
We began this section with the following definition.
\begin{definition}
	A family of adjointable operators $\{T_{w}\}_{w\in\Omega} \subset End_{\mathcal{A}}^{\ast}(\mathcal{H}) $ on a Hilbert $\mathcal{A}$-module $\mathcal{H}$ over a unital $C^{\ast}$-algebra is said to be an integral operator frame for $End_{\mathcal{A}}^{\ast}(\mathcal{H})$, if there exist two positive constants $A, B > 0$ such that 
	\begin{equation}\label{eq333}
		A\langle x,x\rangle_{\mathcal{A}} \leq\int_{\Omega}\langle T_{\omega}x,T_{\omega}x\rangle_{\mathcal{A}} d\mu(\omega)\leq B\langle x,x\rangle_{\mathcal{A}} , \quad x\in\mathcal{H}.
	\end{equation}

	If the sum in the middle of \eqref{eq333} is convergent in norm, the integral operator frame is called standard. 
\end{definition}

\begin{definition}
	Let $K\in End_{\mathcal{A}}^{\ast}(\mathcal{H})$ and $T=\{T_{\omega} \in End_{\mathcal{A}}^{\ast}(\mathcal{H}), \omega\in\Omega\} $. The family $T$ is said an integral $K$-operator frames for $End_{\mathcal{A}}^{\ast}(\mathcal{H})$, if there exist two positive constants $A, B > 0$ such that 
	\begin{equation}\label{eq3}
	A\langle K^{\ast}x,K^{\ast}x\rangle_{\mathcal{A}} \leq\int_{\Omega}\langle T_{\omega}x,T_{\omega}x\rangle_{\mathcal{A}} d\mu(\omega)\leq B\langle x,x\rangle_{\mathcal{A}} ,\quad  x\in\mathcal{H}.
	\end{equation}
	The numbers $A$ and $B$ are called respectively lower and upper bound of the integral $K$-operator frame.\\
	An integral $K$-operator frame $\{T_{\omega}\}_{\omega \in \Omega} \subset End_{\mathcal{A}}^{\ast}(\mathcal{H})$ is said to be $A$-tight if there exists a constant $0 < A$ such that,
\begin{equation*}
A\langle K^{\ast}x,K^{\ast}x\rangle_{\mathcal{A}}=\int_{\Omega} \langle T_{\omega}x,T_{\omega}x\rangle_{\mathcal{A}} d\mu(\omega), \quad x\in \mathcal{H}
\end{equation*}
If $A=1$, it is called a normalized tight integral $K$-operator frames or a Parseval integral $K$-operator frame. 
\end{definition}
	\begin{example}
	Let $\mathcal{H}$ be a Hilbert space defined by:\\
	$\mathcal{H}=\left\{ \left( 
	\begin{array}{ccc}
	\alpha & 0 & 0 \\ 
	0 & 0 & \beta 
	\end{array}%
	\right) \text{ / }\alpha, \beta\in 
	\mathbb{C}
	\right\} $
	and $\mathcal{A}=\left\{ \left( 
	\begin{array}{ccc}
	a & 0 \\ 
	0 & b 
	\end{array}%
	\right) \text{ / }a,b\in 
	\mathbb{C}
	\right\} $. \\
	It's clear that $\mathcal{H}$ is a Hilbert space and $\mathcal{H}\mathcal{A}\subset \mathcal{H}$.\\
	Furthermore, the $\mathcal{A}$-valued inner product,
	\[
	\begin{array}{ccc}
	\mathcal{H}\times \mathcal{H} & \longrightarrow  & \mathcal{A} \\ 
	(A,B) & \longrightarrow  & \langle A, B\rangle_{\mathcal{A}} = A^{t}\bar{B}%
	\end{array}%
	\]\\
	is sesquilinear and positive.\\
	If  
	$ A=\left( 
	\begin{array}{ccc}
	a_{1} & 0 &0  \\ 
	0 & 0 & a_{2} 
	\end{array}%
	\right)$ and 
	$B=\left( 
	\begin{array}{ccc}
	b_{1} & 0 & 0  \\ 
	0 & 0 & b_{2} 
	\end{array}%
	\right)$,\\
	 then, 
	\begin{equation*}
	\langle A,B\rangle_{\mathcal{A}}=\left( 
	\begin{array}{ccc}
	a_{1}\bar{b_{1}} & 0  \\ 
	0 &  a_{2}\bar{b_{2}} 
	\end{array}%
	\right)
	\end{equation*}
	Let $(\Omega=\left[ 0,1\right] ,d\lambda ) $ be a measure space, where $d\lambda  $ is a Lebesgue measure  restraint on the interval $\left[ 0,1\right] $.\\
	For all $w \in \left[ 0,1\right] $, we consider,
	\begin{align*}
	F : \Omega &\longrightarrow \mathcal{H}\\
	w&\longrightarrow F_{\omega}= \left(   \begin{array}{ccc}
	w & 0 & 0 \\ 
	0 & 0 & 0 
	\end{array}
	\right)
	\end{align*}
	$F$ is a measurable map and for all $A\in \mathcal{H}$, we have,
	\begin{align*}
	\int_{\Omega}\langle A,F_{\omega}\rangle_{\mathcal{A}}\langle F_{\omega},A\rangle_{\mathcal{A}}d\lambda(w)&=\int_{\Omega}\left( 
	\begin{array}{ccc}
	a_{1}\bar{w} & 0  \\ 
	0 &  0 
	\end{array}%
	\right)
	\left( 
	\begin{array}{ccc}
	w\bar{a_{1}} & 0  \\ 
	0 &  0 
	\end{array}%
	\right)d\lambda(w)\\
	&=\left( 
	\begin{array}{ccc}
	|a_{1}| & 0  \\ 
	0 &  0 
	\end{array}%
	\right)
	\left( 
	\begin{array}{ccc}
	\int_{\Omega}w^{2}d\lambda(w) & 0  \\ 
	0 &  0 
	\end{array}%
	\right)\\
	&=\frac{1}{3}\left( 
	\begin{array}{ccc}
	|a_{1}|^{2} & 0  \\ 
	0 &  0 
	\end{array}%
	\right)\\
	&\leq \frac{1}{3}\left( 
	\begin{array}{ccc}
	|a_{1}|^{2} & 0  \\ 
	0 &  |a_{2}|^{2} 
	\end{array}%
	\right)=\frac{1}{3}\langle A,A\rangle_{\mathcal{A}}.
	\end{align*}
	Wich show that $F$ is an integral Bessel sequence. But $F$ is not an integral operator frame for the Hilbert $\mathcal{A}$-modules. Indeed, just take $\left( \begin{array}{ccc}
	0 & 0 & 0  \\ 
	0 &  0 & a_{2} 
	\end{array}%
	\right) $ with $a_{2}\neq 0$.\\
	Consider now, 
	\begin{align*}
	K : \mathcal{H}\qquad  &\longrightarrow \qquad \mathcal{H}\\
	\left( 
	\begin{array}{ccc}
	a_{1} & 0 &0  \\ 
	0 &  0 &a_{2}
	\end{array}%
	\right)&\longrightarrow \left( 
	\begin{array}{ccc}
	a_{1} & 0 &0  \\ 
	0 &  0 & 0
	\end{array}%
	\right)
	\end{align*}
	$K$ is a linear, bounded and selfadjoint operator and we have for all $A\in \mathcal{H}$,
	\begin{equation*}
	\langle K^{\ast}A,K^{\ast}A\rangle_{\mathcal{A}}=\left( 
		\begin{array}{ccc}
		|a_{1}|^{2} & 0  \\ 
		0 &  0 
		\end{array}%
		\right)
	\end{equation*}
	So,
	\begin{equation*}
	\frac{1}{4}\langle K^{\ast}A,K^{\ast}A\rangle_{\mathcal{A}}\leq \int_{\Omega}\langle A,F_{\omega}\rangle_{\mathcal{A}}\langle F_{\omega},A\rangle_{\mathcal{A}}d\lambda(w)\leq \frac{1}{3}\langle A,A\rangle_{\mathcal{A}}.
	\end{equation*}

\end{example}
\begin{remark}
	Every integral operator frame is an integral $K$-operator frame, for any $K\in End_{\mathcal{A}}^{\ast}(\mathcal{H})$, $K\neq0$.\\
	 Indeed, if $\{T_{\omega}\}_{\omega\in\Omega}$ is an integral operator for $ End_{\mathcal{A}}^{\ast}(\mathcal{H})$with bounds $A$ and $B$. then
	\begin{equation*}
	 A\langle x,x\rangle_{\mathcal{A}}\leq\int_{\Omega} \langle T_{w}x,T_{\omega}x\rangle_{\mathcal{A}} d\mu(\omega)\leq B\langle x,x\rangle_{\mathcal{A}},\quad   x\in\mathcal{H}.
	\end{equation*}
	By lemma \ref{l1}, we have,
	\begin{equation*}
	A\|K\|^{-2}\langle K^{\ast}x,K^{\ast}x\rangle_{\mathcal{A}}\leq \int_{\Omega} \langle T_{\omega}x,T_{\omega}x\rangle_{\mathcal{A}} d\mu(\omega)\leq B\langle x,x\rangle_{\mathcal{A}}, \quad x\in\mathcal{H}.
	\end{equation*}
	Therefore the family $\{T_{\omega}\}_{\omega\in\Omega}$ is an integral $K$-operator frame with bounds $A\|K\|^{-2}$ and $B$.
\end{remark}
\begin{proposition}
	Let $K\in End_{\mathcal{A}}^{\ast}(\mathcal{H})$ and $\{T_{\omega}\}_{\omega\in\Omega}$ be an integral $K$-operator frame for $End_{\mathcal{A}}^{\ast}(\mathcal{H})$ with frame bounds $A$ and $B$. If $K$ is surjective then $\{T_{\omega}\}_{\omega\in\Omega}$ is an integral operator frame for $End_{\mathcal{A}}^{\ast}(\mathcal{H})$.
\end{proposition}
\begin{proof}
	Since $K$ is surjective, there exists $m>0$ such that
	\begin{equation*}
	\langle K^{\ast}x,K^{\ast}x\rangle_{\mathcal{A}}\geq m\langle x,x\rangle_{\mathcal{A}},  x\in\mathcal{H}.
	\end{equation*}
	Also, since $\{T_{\omega}\}_{\omega\in\Omega}$ is an integral  $K$-operator frame for $End_{\mathcal{A}}^{\ast}(\mathcal{H})$, we have
	\begin{equation*}
	Am\langle x,x\rangle_{\mathcal{A}}\leq A\langle K^{\ast}x,K^{\ast}x\rangle_{\mathcal{A}}\leq\int_{\Omega}\langle T_{\omega}x,T_{\omega}x\rangle_{\mathcal{A}} d\mu(\omega)\leq B\langle x,x\rangle_{\mathcal{A}}, \quad x\in\mathcal{H}.
	\end{equation*}
	Hence $\{T_{\omega}\}_{\omega\in\Omega}$ is an integral operator frame for $End_{\mathcal{A}}^{\ast}(\mathcal{H})$ with frame bounds $Am$ and $B$.
\end{proof}
Let $\{T_{\omega}\}_{\omega\in\Omega}$ be an integral $K$-operator frame for $End_{\mathcal{A}}^{\ast}(\mathcal{H})$.\\
we define the operator $R$ by,
\begin{align*}
R:\mathcal{H}&\longrightarrow l^{2}(\Omega,\mathcal{H})\\
x&\longrightarrow Rx=\{T_{\omega}x\}_{\omega\in\Omega},  
\end{align*}
The operator $R$ is called the analysis operator of the integral $K$-operator frame $\{T_{\omega}\}_{\omega\in\Omega} $.\\
The adjoint of the analysis operator $R$ is defined by,
\begin{align*}
R^{\ast}:l^{2}(\Omega,\mathcal{H})&\longrightarrow \mathcal{H}\\
x&\longrightarrow R^{\ast}(\{x_{\omega}\}_{\omega\in\Omega})=\int_{\Omega}T_{\omega}^{\ast}x_{\omega}d\mu(\omega),  
\end{align*}
The operator $R^{\ast}$ is called the synthesis operator of the integral $K$-operator frame $ \{T_{\omega}\}_{\omega\in\Omega} $.\\
By composing $R$ and $R^{\ast}$, the frame operator $S_{T}:\mathcal{H}\rightarrow\mathcal{H}$ for the integral $K$-operator frame $T$ is given by
\begin{equation*}
S(x)=R^{\ast}Rx=\int_{\Omega}T_{\omega}^{\ast}T_{\omega}xd\mu(\omega)
\end{equation*} 

\begin{theorem}
	Let $K\in End_{\mathcal{A}}^{\ast}(\mathcal{H})$ and $\{T_{\omega}\}_{\omega\in\Omega}\subset End_{\mathcal{A}}^{\ast}(\mathcal{H})$. The following
	statements are equivalent:
	\begin{itemize}
		\item [(1)] $\{T_{\omega}\}_{\omega\in\Omega}$ is an integral $K$-operator frame for $End_{\mathcal{A}}^{\ast}(\mathcal{H})$.
		\item [(2)] There exists $A>0$ such that $ AKK^{\ast}\leq S$, where $S$ is the frame operator for $\{T_{\omega}\}_{\omega\in\Omega}$.
		\item[(3)] $K=S^{\frac{1}{2}}Q$, for some $Q\in End_{\mathcal{A}}^{\ast}(\mathcal{H})$.
	\end{itemize}
\end{theorem}
\begin{proof}
	$(1)\Rightarrow(2)$ Let $\{T_{\omega}\}_{\omega\in\Omega}$ be an integral $K$-operator frame for $End_{\mathcal{A}}^{\ast}(\mathcal{H})$ with frame
	bounds $A$, $B$ and frame operator $S$ if and only if
	\begin{displaymath}
	A\langle K^{\ast}x,K^{\ast}x\rangle_{\mathcal{A}}\leq\int_{\Omega}\langle T_{\omega}x,T_{\omega}x\rangle_{\mathcal{A}} d\mu(\omega)\leq B\langle x,x\rangle_{\mathcal{A}},\quad  x\in\mathcal{H}.
	\end{displaymath}
	Then
	\begin{displaymath}
	\langle AKK^{\ast}x,x\rangle_{\mathcal{A}}\leq\langle Sx,x\rangle_{\mathcal{A}}\leq\langle Bx,x\rangle_{\mathcal{A}},  \quad x\in\mathcal{H}.
	\end{displaymath}
	Hence 
	\begin{equation}
 AKK^{\ast} \leq S.
	\end{equation}	
	$(2)\Rightarrow(3)$ Suppose there exists $A>0$ such that $AKK^{\ast}\leq S$. \\
	Notice that
	\begin{equation*}
	AKK^{\ast}\leq (S^{\frac{1}{2}})(S^{\frac{1}{2}})^{\ast}
	\end{equation*}
	By Lemma \ref{l2} we have:\\
	 $K=S^{\frac{1}{2}}Q$, for some $Q\in End_{\mathcal{A}}^{\ast}(\mathcal{H})$.\\
	$(3)\Rightarrow(1)$ Suppose that $K=S^{\frac{1}{2}}Q$, for some $Q\in End_{\mathcal{A}}^{\ast}(\mathcal{H})$.\\
	Then by the Lemma \ref{l2}, there exists $A>0$ such that $AKK^{\ast}\leq (S^{\frac{1}{2}})(S^{\frac{1}{2}})^{\ast}$.\\ So,  $AKK^{\ast}\leq S$.\\
	Which shows that $\{T_{\omega}\}_{\omega\in \Omega}$ is an integral $K$-operator frame for $End_{\mathcal{A}}^{\ast}(\mathcal{H})$.
\end{proof}
\section{Some Constructions of Integral $K$-Operator frame}
\begin{theorem}
	Let $Q\in End_{\mathcal{A}}^{\ast}(\mathcal{H})$ and let  $\{T_{\omega}\}_{\omega\in\Omega}$ be an integral $K$-operator frame for $End_{\mathcal{A}}^{\ast}(\mathcal{H})$. Then
	$\{T_{\omega}Q\}_{\omega\in\Omega}$ is an integral $(Q^{\ast}K)$-operator frame for $End_{\mathcal{A}}^{\ast}(\mathcal{H})$.
\end{theorem}
\begin{proof}
	Let $\{T_{\omega}\}_{\omega\in\Omega}$ be an integral $K$-operator frame for $End_{\mathcal{A}}^{\ast}(\mathcal{H})$ with frame
	bounds $A$ and $B$ if and only if,
	\begin{displaymath}
	A\langle K^{\ast}x,K^{\ast}x\rangle_{\mathcal{A}}\leq\int_{\Omega}\langle T_{\omega}x,T_{\omega}x\rangle_{\mathcal{A}} d\mu(\omega)\leq B\langle x,x\rangle_{\mathcal{A}}, \quad  x\in\mathcal{H}.
	\end{displaymath}
	This give,
	\begin{displaymath}
	A\langle K^{\ast}Qx,K^{\ast}Qx\rangle_{\mathcal{A}}\leq\int_{\Omega}\langle T_{\omega}Qx,T_{\omega}Qx\rangle_{\mathcal{A}} d\mu(\omega)\leq B\langle Qx,Qx\rangle_{\mathcal{A}}, \quad  x\in\mathcal{H}.
	\end{displaymath}
	So,
	\begin{displaymath}
	A\langle (Q^{\ast}K)^{\ast}x,(Q^{\ast}K)^{\ast}x\rangle_{\mathcal{A}}\leq\int_{\Omega}\langle T_{\omega}Qx,T_{\omega}Qx\rangle_{\mathcal{A}} d\mu(\omega)\leq B\|Q\|^{2}\langle x,x\rangle_{\mathcal{A}}, \quad  x\in\mathcal{H}.
	\end{displaymath}
which shows that $\{T_{\omega}Q\}_{\omega\in\Omega}$ is an integral $(Q^{\ast}K)$-operator frame for $End_{\mathcal{A}}^{\ast}(\mathcal{H})$ with bounds $A$ and $B\|Q\|^{2}$.
\end{proof}
\begin{theorem}
	Let $K\in End_{\mathcal{A}}^{\ast}(\mathcal{H})$ and let $\{T_{\omega}\}_{\omega\in\Omega}\subset End_{\mathcal{A}}^{\ast}(\mathcal{H})$ be a tight integral  $K$-operator frame for $End_{\mathcal{A}}^{\ast}(\mathcal{H})$ with frame bound $A_{1}$. Then $\{T_{\omega}\}_{\omega\in\Omega}$ is a tight integral operator frame for $End_{\mathcal{A}}^{\ast}(\mathcal{H})$ with frame bound $A_{2}$ if and only if $K_{r}^{-1}=\frac{A_{1}}{A_{2}}K^{\ast}$.
\end{theorem}
\begin{proof}
	Let $\{T_{\omega}\}_{\omega\in\Omega}\subset End_{\mathcal{A}}^{\ast}(\mathcal{H})$ be a tight integral $K$-operator frame for $End_{\mathcal{A}}^{\ast}(\mathcal{H})$ with frame bound $A_{1}$, then 
	\begin{equation*}
\int_{\Omega}\langle T_{\omega}x,T_{\omega}x\rangle_{\mathcal{A}} d\mu(\omega)=A_{1}\langle K^{\ast}x,K^{\ast}x\rangle_{\mathcal{A}},\quad  x\in\mathcal{H};
	\end{equation*}
	Since $\{T_{\omega}\}_{\omega\in\Omega}$ is a tight integral operator frame for $End_{\mathcal{A}}^{\ast}(\mathcal{H})$ with frame bound $A_{2}$, then,
	\begin{displaymath}
	\int_{\Omega}\langle T_{\omega}x,T_{\omega}x\rangle_{\mathcal{A}} d\mu(\omega)=A_{2}\langle x,x\rangle_{\mathcal{A}},\quad  x\in\mathcal{H};
	\end{displaymath}
	We deduce that, for each $x\in\mathcal{H}$, we have
	\begin{equation*}
A_{1}\langle K^{\ast}x,K^{\ast}x\rangle_{\mathcal{A}}=A_{2}\langle x,x\rangle_{\mathcal{A}}
	\end{equation*} 
	So, 
	\begin{displaymath}
	\langle KK^{\ast}x,x\rangle_{\mathcal{A}}=\langle \frac{A_{2}}{A_{1}}x,x\rangle_{\mathcal{A}}, \quad  x\in\mathcal{H}.
	\end{displaymath}
	Then $KK^{\ast}=\frac{A_{2}}{A_{1}}I$, Hence $K_{r}^{-1}=\frac{A_{1}}{A_{2}}K^{\ast}$.\\
	Conversely, suppose that $K_{r}^{-1}=\frac{A_{1}}{A_{2}}K^{\ast}$. Then $KK^{\ast}=\frac{A_{2}}{A_{1}}I$. Thus
	\begin{displaymath}
	\langle KK^{\ast}x,x\rangle_{\mathcal{A}}=\langle \frac{A_{2}}{A_{1}}x,x\rangle_{\mathcal{A}},  \quad x\in\mathcal{H}.
	\end{displaymath}
	Since $\{T_{\omega}\}_{\omega\in \Omega}$ is a tight $K$-operator frame for $End_{\mathcal{A}}^{\ast}(\mathcal{H})$, we have
	\begin{displaymath}
	\int_{\Omega}\langle T_{\omega}x,T_{\omega}x\rangle_{\mathcal{A}} d\mu(\omega)=A_{2}\langle x,x\rangle_{\mathcal{A}}, \quad  x\in\mathcal{H}
	\end{displaymath}
which ends the proof.
\end{proof}
\begin{theorem}
	Let $\{T_{\omega}\}_{\omega\in\Omega}$ be an integral $K$-operator frame for $\mathcal{H}$ with best frame bounds $A$ and $B$. If $Q\in End_{\mathcal{A}}^{\ast}(\mathcal{H})$ be an adjointable and invertible operator such that $Q^{-1}K^{\ast}=K^{\ast}Q^{-1}$, then $\{T_{\omega}Q\}_{\omega\in\Omega}$ is an integral $K$-operator frame for $\mathcal{H}$ with best frame bounds $C$ and $D$ satisfying the inequalities
	\begin{equation}\label{eq7}
	A\|Q^{-1}\|^{-2}\leq C\leq A\|Q\|^{2} \;\; and \;\; B\|Q^{-1}\|^{-2}\leq D\leq B\|Q\|^{2}
	\end{equation}
\end{theorem}
\begin{proof}
	Let $\{T_{\omega}\}_{\omega\in\Omega}$ be an integral $K$-operator frame for $\mathcal{H}$ with best frame bounds $A$ and $B$.\\
	On one hand, we have  for all $x\in \mathcal{H}$,
	
	\begin{displaymath}
	\int_{\Omega}\langle T_{\omega}Qx,T_{\omega}Qx\rangle_{\mathcal{A}} d\mu(\omega)\leq B\langle Qx,Qx\rangle_{\mathcal{A}}\leq B\|Q\|^{2}\langle x,x\rangle_{\mathcal{A}}.
	\end{displaymath}
	One other hand, we have for all $x\in \mathcal{H}$,
	\begin{align*}
	A\langle K^{\ast}x,K^{\ast}x\rangle_{\mathcal{A}}&=A\langle K^{\ast}Q^{-1}Qx,K^{\ast}Q^{-1}Qx\rangle_{\mathcal{A}}\\
	&=A\langle Q^{-1}K^{\ast}Qx,Q^{-1}K^{\ast}Qx\rangle_{\mathcal{A}}\\
	&\leq\|Q^{-1}\|^{2}\int_{\Omega}\langle T_{\omega}Qx,T_{\omega}Qx\rangle_{\mathcal{A}} d\mu(\omega).
	\end{align*}
	So, we conclude,
	\begin{equation*}
	A\|Q^{-1}\|^{-2}\langle K^{\ast}x,K^{\ast}x\rangle_{\mathcal{A}}\leq\int_{\Omega}\langle T_{\omega}Qx,T_{\omega}Qx\rangle_{\mathcal{A}} d\mu(\omega)\leq B\|Q\|^{2}\langle x,x\rangle_{\mathcal{A}}
	\end{equation*}
	Wich show that $\{T_{\omega}Q\}_{\omega\in\Omega}$ is an integral $K$-operator frame for $End_{\mathcal{A}}^{\ast}(\mathcal{H})$ with bounds $A\|Q^{-1}\|^{-2}$ and $B\|Q\|^{2}$.
	Now, let $E$ and $F$ be the best bounds of the integral $K$-operator frame $\{T_{\omega}Q\}_{\omega\in\Omega}$. Then
	\begin{equation}\label{eq8}
	A\|Q^{-1}\|^{-2}\leq E \;\; and \;\; F\leq B\|Q\|^{2}
	\end{equation}
	Since $\{T_{\omega}Q\}_{\omega\in\Omega}$ is an integral $K$-operator frame for $End_{\mathcal{A}}^{\ast}(\mathcal{H})$ with frame bounds $E$ and $F$ and
	\begin{equation*}
	\langle K^{\ast}x,K^{\ast}x\rangle_{\mathcal{A}}=\langle QQ^{-1}K^{\ast}x,QQ^{-1}K^{\ast}x\rangle_{\mathcal{A}}\leq\|Q\|^{2}\langle K^{\ast}Q^{-1}x,K^{\ast}Q^{-1}x\rangle_{\mathcal{A}}.
	\end{equation*}
	Hence
	\begin{align*}
	E\|Q\|^{-2}\langle K^{\ast}x,K^{\ast}x\rangle_{\mathcal{A}}&\leq E\langle K^{\ast}Q^{-1}x,K^{\ast}Q^{-1}x\rangle_{\mathcal{A}}\\
	&\leq \int_{\Omega}\langle T_{\omega}QQ^{-1}x,T_{\omega}QQ^{-1}x\rangle_{\mathcal{A}} d\mu(\omega)\\
	&=\int_{\Omega}\langle T_{\omega}x,T_{\omega}x\rangle_{\mathcal{A}} d\mu(\omega)\\
	&\leq F\|Q^{-1}\|^{2}\langle x,x\rangle_{\mathcal{A}}.
	\end{align*}
	Since $A$ and $B$ are the best bounds of integral $K$-operator frame $\{T_{\omega}\}_{\omega\in\Omega}$, we have
	\begin{equation}\label{eq9}
	E\|Q\|^{-2}\leq A \;\; and \;\; B\leq F\|Q^{-1}\|^{2}.
	\end{equation}
	Which ends the proof.
	\end{proof}
\section{Operator Preserving Integral $K$-Operator Frame}
\begin{proposition}
	Let $K,T\in End_{\mathcal{A}}^{\ast}(\mathcal{H})$ such that $R(T) \subset R(K)$ and $R(K)$ is closed. Let $\{T_{\omega}\}_{\omega\in\Omega}$  be an integral $K$-operator frame for $End_{\mathcal{A}}^{\ast}(\mathcal{H})$, Then $\{T_{\omega}\}_{\omega\in\Omega}$  is an integral $T$-operator frame for $End_{\mathcal{A}}^{\ast}(\mathcal{H})$.
\end{proposition}
\begin{proof}
	Suppose that $\{T_{\omega}\}_{\omega\in\Omega}$ is an integral $K$-operator frame for $End_{\mathcal{A}}^{\ast}(\mathcal{H})$. Then there exist two positive constants A and B such that 
\begin{equation}\label{h1}
    A\langle K^{\ast}x,K^{\ast}x\rangle_{\mathcal{A}}\leq\int_{\Omega}\langle T_{\omega}x,T_{\omega}x\rangle_{\mathcal{A}} d\mu(\omega)\leq B\langle x,x\rangle_{\mathcal{A}}, \quad  x\in\mathcal{H}.
\end{equation}
   From lemma \ref{l2}, there exist $0<\lambda$ such that 
    $$TT^{\ast}\leq \lambda KK^{\ast}.$$
    Which give,
    $$\frac{A}{\lambda}\langle T^{\ast}x,T^{\ast}x\rangle_{\mathcal{A}}\leq A \langle K^{\ast}x,K^{\ast}x\rangle_{\mathcal{A}}\leq \int_{\Omega}\langle T_{\omega}x,T_{\omega}x\rangle_{\mathcal{A}} d\mu(\omega) \leq B\langle x,x\rangle_{\mathcal{A}}, \quad  x\in\mathcal{H}.$$
    Hence $\{T_{\omega}\}_{\omega\in\Omega}$ is an integral $T$-operator frame for $End_{\mathcal{A}}^{\ast}(\mathcal{H})$. 
\end{proof}
\begin{theorem}
	Let $K\in End_{\mathcal{A}}^{\ast}(\mathcal{H})$ with a dense range. Let $\{T _{\omega}\}_{\omega\in\Omega}$ be an integral $K$-operator frame for $End_{\mathcal{A}}^{\ast}(\mathcal{H})$ and $T\in End_{\mathcal{A}}^{\ast}(\mathcal{H})$ have closed range and commute with $T_{\omega}$ for each $\omega \in \Omega$. If $\{TT_{\omega}\}_{\omega\in\Omega}$ is an integral $K$-operator frame for $End_{\mathcal{A}}^{\ast}(\mathcal{H})$ then T is surjective.
\end{theorem}
\begin{proof}
	Assume that $\{TT_{\omega}\}_{\omega\in\Omega}$ is an integral $K$-operator frame for $End_{\mathcal{A}}^{\ast}(\mathcal{H})$ with bounds A and B, then 
\begin{equation}\label{h2}
    A\langle K^{\ast}x,K^{\ast}x\rangle_{\mathcal{A}}\leq\int_{\Omega}\langle TT_{\omega}x,TT_{\omega}x\rangle_{\mathcal{A}} d\mu(\omega)\leq B\langle x,x\rangle_{\mathcal{A}},\quad   x\in\mathcal{H}.
\end{equation}
	Since K is dense range, then $K^{\ast}$ is injective.\\
	 By \eqref{h2}, $T^{\ast}$ is injective since, $N(T^{\ast})\subset N(K^{\ast})$.\\
	Moreover, $R(T)=N(T^{\ast})^{\perp}=\mathcal{H}$, which show that $T$ is surjective.
\end{proof}
\begin{theorem}
Let $K,T\in End_{\mathcal{A}}^{\ast}(\mathcal{H})$ and let $\{T _{\omega}\}_{\omega\in\Omega}$ be an integral $K$-operator frame for $End_{\mathcal{A}}^{\ast}(\mathcal{H})$. If $T$ has closed range and commutes with $K^{\ast}$ and commue with $T_{\omega}$ for each $\omega \in \Omega$, then $\{TT _{w}\}_{\omega\in\Omega}$ is an integral $K$-operator frame for $R(T)$.
\end{theorem} 
\begin{proof}
let $\{T _{\omega}\}_{\omega\in\Omega}$ be an integral $K$-operator frame with bounds $A$,$B$.\\
If $T$ has closed range, it has the pseudo-inverse $T^{\dagger}$ such that $T^{\dagger}T=I$.\\
So, we have for all $x\in R(T)$,
\begin{equation*}
\langle K^{\ast}x,K^{\ast}x\rangle_{\mathcal{A}} =\langle T^{\dagger}TK^{\ast}x,T^{\dagger}TK^{\ast}x\rangle_{\mathcal{A}}\leq \|T^{\dagger}\|^{2}\langle TK^{\ast}x,TK^{\ast}x\rangle_{\mathcal{A}}
\end{equation*}
So,
\begin{equation*}
\|T^{\dagger}\|^{-2}\langle K^{\ast}x,K^{\ast}x\rangle_{\mathcal{A}}\leq \langle TK^{\ast}x,TK^{\ast}x\rangle_{\mathcal{A}}.
\end{equation*}
On one hand, for each $x\in R(T)$, we have,
\begin{align*}
\int_{\Omega}\langle TT_{\omega}x,TT_{\omega}x\rangle_{\mathcal{A}} d\mu(\omega)=\int_{\Omega}\langle T_{\omega}Tx,T_{\omega}Tx\rangle_{\mathcal{A}} d\mu(\omega)&\geq A\langle K^{\ast}Tx,K^{\ast}Tx\rangle_{\mathcal{A}}\\
&=A\langle TK^{\ast}x,TK^{\ast}x\rangle_{\mathcal{A}}\\
&\geq A\|T^{\dagger}\|^{-2}\langle K^{\ast}x,K^{\ast}x\rangle_{\mathcal{A}}
\end{align*}
One other hand, we have,
\begin{align*}
\int_{\Omega}\langle TT_{\omega}x,TT_{\omega}x\rangle_{\mathcal{A}} d\mu(\omega)=\int_{\Omega}\langle T_{\omega}Tx,T_{\omega}Tx\rangle_{\mathcal{A}} d\mu(\omega)&\leq B\langle Tx,Tx\rangle_{\mathcal{A}}\\
&=B\|T\|^{2}\langle x,x\rangle_{\mathcal{A}}\\
\end{align*}
which shows that $\{TT _{\omega}\}_{\omega\in\Omega}$ is an integral $K$-operator frame for $R(T)$ with bounds $A\|T^{\dagger}\|^{-2}$ and $B\|T\|^{2}$.
\end{proof} 

	\section{Perturbation of Integral $K$-Operator Frames}
In this section we consider perturbation of an integral $K$-operator frames by non-zero operators. 
\begin{theorem}
	Let $K \in End_{\mathcal{A}}^{\ast}(\mathcal{H})$ and $\{T_{\omega}\}_{\omega \in \Omega}$ be an integral $K$-operator frame for $End_{\mathcal{A}}^{\ast}(\mathcal{H})$ with frames bounds $A$ and $B$. Let $ L \in End_{\mathcal{A}}^{\ast}(\mathcal{H}), (L\neq 0)$, and $\{a_{\omega}\}_{\omega \in \Omega}$ any family of scalars. Then the perturbed family of operator $\{T_{\omega} + a_{\omega}LK^{\ast}\}_{\omega \in \Omega}$ is an integral $K$-operator frames for $End_{\mathcal{A}}^{\ast}(\mathcal{H})$ if $\int_{\Omega}|a_{\omega}|^{2}d\mu(\omega) < \frac{A}{\|L\|}$.
\end{theorem}
\begin{proof}
	Let $\Gamma_{\omega}= T_{\omega} + a_{\omega}LK^{\ast}$,  for all $\omega \in \Omega$. Then for all $x\in \mathcal{H}$, we have,
	\begin{align*}
	\int_{\Omega}\langle T_{\omega}x-\Gamma_{\omega}x,T_{\omega}x-\Gamma_{\omega}x\rangle_{\mathcal{A}} d\mu(\omega)&=\int_{\Omega}\langle a_{\omega}LK^{\ast}x,a_{\omega}LK^{\ast}x\rangle_{\mathcal{A}} d\mu(\omega),\\
	&\leq \int_{\Omega}|a_{\omega}|^{2}\|L\|^{2}\|K^{\ast}\|^{2}\langle x,x\rangle_{\mathcal{A}} d\mu(\omega)\\
	&=\int_{\Omega}|a_{\omega}|^{2}\|L\|^{2}\|K^{\ast}\|^{2}\langle x,x\rangle_{\mathcal{A}} d\mu(\omega),\\
	&\leq R\|K^{\ast}\|^{2}\langle x,x\rangle_{\mathcal{A}}. \qquad where \qquad  R=\int_{\Omega}|a_{\omega}|^{2}\|L\|^{2}d\mu(\omega).
	\end{align*}
	On one hand, for all $x\in H$,we have, 
	\begin{align*}
	(\int_{\Omega}\langle (T_{\omega} + a_{\omega}LK^{\ast})x,(T_{\omega} + a_{\omega}LK^{\ast})x\rangle_{\mathcal{A}} d\mu(\omega))^{\frac{1}{2}}&=\|(T_{\omega} + a_{\omega}LK^{\ast})x\|_{l^{2}(\Omega,H)}\\
	&\leq \|T_{\omega}x\|_{l^{2}(\Omega,H)} + \| a_{\omega}LK^{\ast}x\|_{l^{2}(\Omega,H)}\\
	&= (\int_{\Omega}\langle T_{\omega}x,T_{\omega}x\rangle_{\mathcal{A}} d\mu(\omega))^{\frac{1}{2}}  \\
	&  + (\int_{\Omega}\langle (a_{\omega}LK^{\ast})x,(a_{\omega}LK^{\ast})x\rangle_{\mathcal{A}} d\mu(\omega))^{\frac{1}{2}}\\
	&\leq \sqrt{B}\langle x,x\rangle_{\mathcal{A}} +  \sqrt{R}\|K^{\ast}\|\langle x,x\rangle_{\mathcal{A}}\\
	&\leq (\sqrt{B} + \sqrt{R}\|K^{\ast}\|)\langle x,x\rangle_{\mathcal{A}}.
	\end{align*}
	Then 
	\begin{equation}\label{123}
	\int_{\Omega}\langle (T_{\omega} + a_{\omega}LK^{\ast})x,(T_{\omega} + a_{\omega}LK^{\ast})x\rangle_{\mathcal{A}} d\mu(\omega)\leq (\sqrt{B} + \sqrt{R}\|K^{\ast}\|)^{2}\langle x,x\rangle_{\mathcal{A}}.
	\end{equation}
	On other hand, for all $x\in H$, we have,
	\begin{align*}
	(\int_{\Omega}\langle (T_{\omega} + a_{\omega}LK^{\ast})x,(T_{\omega} + a_{\omega}LK^{\ast})x\rangle_{\mathcal{A}} d\mu(\omega))^{\frac{1}{2}}&=\|(T_{\omega} + a_{\omega}LK^{\ast})x\|_{l^{2}(\Omega,\mathcal{H})}\\
	&\geq \|T_{\omega}x\|_{l^{2}(\Omega,\mathcal{H})} - \| a_{\omega}LK^{\ast}x\|_{l^{2}(\Omega,\mathcal{H})}\\
	&\geq (\int_{\Omega}\langle T_{\omega}x,T_{\omega}x\rangle_{\mathcal{A}} d\mu(\omega))^{\frac{1}{2}} \\
	& \textcolor{white}{.....} - (\int_{\Omega}\langle a_{\omega}LK^{\ast}x,a_{\omega}LK^{\ast}x\rangle_{\mathcal{A}} d\mu(\omega))^{\frac{1}{2}}\\
	&\geq \sqrt{A}\langle K^{\ast}x, K^{\ast}x\rangle_{\mathcal{A}} -  \sqrt{R}\langle K^{\ast}x, K^{\ast}x\rangle_{\mathcal{A}} \\
	&\geq (\sqrt{A} -  \sqrt{R})\langle K^{\ast}x, K^{\ast}x\rangle_{\mathcal{A}} 
	\end{align*}
	So,
	\begin{equation}\label{124}
	\int_{\Omega}\langle (T_{\omega} + a_{\omega}LK^{\ast})x,(T_{\omega} + a_{\omega}LK^{\ast})x\rangle_{\mathcal{A}} d\mu(\omega) \geq  (\sqrt{A} -  \sqrt{R})^{2}\langle K^{\ast}x, K^{\ast}x\rangle_{\mathcal{A}} 
	\end{equation}
	From \eqref{123} and \eqref{124} we conclude that  $\{T_{\omega} + a_{\omega}LK^{\ast}\}_{\omega \in \Omega}$ is an integral $K$-operator frame for $End_{\mathcal{A}}^{\ast}(\mathcal{H})$ if $R<A$, that is , if :
	\begin{equation*}
	\int_{\Omega}|a_{\omega}|^{2}d\mu(\omega) < \frac{A}{\|L\|}.
	\end{equation*}
	
\end{proof}
\begin{theorem}
	Let $K \in End_{\mathcal{A}}^{\ast}(\mathcal{H})$ and $\{T_{\omega}\}_{\omega \in \Omega}$ be an integral $K$-operator frame for $End_{\mathcal{A}}^{\ast}(\mathcal{H})$. Let $\{\Gamma_{\omega}\}_{\omega \in \Omega}$ be any family on $End_{\mathcal{A}}^{\ast}(\mathcal{H})$, and $\{a_{\omega}\}_{\omega \in \Omega},\, \{b_{\omega}\}_{\omega \in \Omega} \subset \mathbb{R}$ be two positively  confined sequences. If there exists a constants $\alpha , \beta$  with $0\leq \alpha , \beta<\frac{1}{2}$ such that,
	\begin{align}\label{1000}
	\int_{\Omega}\langle a_{\omega}T_{\omega}x - b_{\omega}\Gamma_{\omega}x,a_{\omega}T_{\omega}x - b_{\omega}\Gamma_{\omega}x\rangle_{\mathcal{A}} d\mu(\omega)&\leq \alpha\int_{\Omega}\langle a_{\omega}T_{\omega}x,a_{\omega}T_{\omega}x\rangle_{\mathcal{A}} d\mu(\omega) \\
	&\textcolor{white}{.....} + \beta\int_{\Omega}\langle b_{\omega}\Gamma_{\omega}x,b_{\omega}\Gamma_{\omega}x\rangle_{\mathcal{A}} d\mu(\omega).
	\end{align}
	Then $\{\Gamma_{\omega}\}_{\omega \in \Omega}$ is an integral $K$-operator frame for $End_{\mathcal{A}}^{\ast}(\mathcal{H})$.
\end{theorem}
\begin{proof}
	Suppose \eqref{1000} holds for some conditions of theorem.\\
	Then for all $x\in H$ we have,
	\begin{align*}
	\int_{\Omega}\langle b_{\omega}\Gamma_{\omega}x,b_{\omega}\Gamma_{\omega}x\rangle_{\mathcal{A}} d\mu(\omega)&\leq 2(\int_{\Omega}\langle a_{\omega}T_{\omega}x,a_{\omega}T_{\omega}x\rangle_{\mathcal{A}} d\mu(\omega)\\
	&\textcolor{white}{.....} +\int_{\Omega}\langle a_{\omega}T_{\omega}x - b_{\omega}\Gamma_{\omega}x,a_{\omega}T_{\omega}x - b_{\omega}\Gamma_{\omega}x\rangle_{\mathcal{A}} d\mu(\omega) ) \\
	&\leq 2(\int_{\Omega}\langle a_{\omega}T_{\omega}x,a_{\omega}T_{\omega}x\rangle_{\mathcal{A}} d\mu(\omega) +\alpha\int_{\Omega}\langle a_{\omega}T_{\omega}x,a_{\omega}T_{\omega}x\rangle_{\mathcal{A}} d\mu(\omega)\\
	&\textcolor{white}{.....} + \beta\int_{\Omega}\langle b_{\omega}\Gamma_{\omega}x, b_{\omega}\Gamma_{\omega}x\rangle_{\mathcal{A}} d\mu(\omega)).
	\end{align*}
	Therefore, 
	\begin{equation*}
	(1-2\beta)\int_{\Omega}\langle b_{\omega}\Gamma_{\omega}x, b_{\omega}\Gamma_{\omega}x\rangle_{\mathcal{A}} d\mu(\omega))\leq 2(1+\alpha)\int_{\Omega}\langle a_{\omega}T_{\omega}x,a_{\omega}T_{\omega}x\rangle_{\mathcal{A}} d\mu(\omega).
	\end{equation*}
	This give,
	\begin{equation*}
	(1-2\beta)[\underset{\omega \in \Omega}{\inf}( b_{\omega})]^{2}\int_{\Omega}\langle \Gamma_{\omega}x,\Gamma_{\omega}x\rangle_{\mathcal{A}} d\mu(\omega)\leq 2(1+\alpha)[\underset{\omega \in \Omega}{\sup}( a_{\omega})]^{2}\int_{\Omega}\langle T_{\omega}x,T_{\omega}x\rangle_{\mathcal{A}} d\mu(\omega).
	\end{equation*}
	Thus,
	\begin{equation}\label{101}
	\int_{\Omega}\langle \Gamma_{\omega}x,\Gamma_{\omega}x\rangle_{\mathcal{A}} d\mu(\omega)\leq \frac{ 2(1+\alpha)[\underset{\omega \in \Omega}{\sup}( a_{\omega})]^{2}}{(1-2\beta)[\underset{\omega \in \Omega}{\inf}( b_{\omega})]^{2}}\int_{\Omega}\langle T_{\omega}x,T_{\omega}x\rangle_{\mathcal{A}} d\mu(\omega).
	\end{equation}
	Also, for all $ x\in H$,
	\begin{align*}
	\int_{\Omega}\langle a_{\omega}T_{\omega}x,a_{\omega}T_{\omega}x\rangle_{\mathcal{A}}d\mu(\omega)&\leq 2(\int_{\Omega}\|a_{\omega}T_{\omega}x - b_{\omega}\Gamma_{\omega}x\|^{2}d\mu(\omega) + \int_{\Omega}\|b_{\omega}\Gamma_{\omega}x\|^{2}d\mu(\omega))\\
	&\leq 2(\alpha\int_{\Omega}\|a_{\omega}T_{\omega}x\|^{2}d\mu(\omega) + \beta\int_{\Omega}\|b_{\omega}\Gamma_{\omega}x\|^{2}d\mu(\omega)+\int_{\Omega}\|b_{\omega}\Gamma_{\omega}x\|^{2}d\mu(\omega)).
	\end{align*}
	Therefore,
	\begin{equation*}
	(1-2\alpha)[\underset{\omega \in \Omega}{\inf}( a_{\omega})]^{2}\int_{\Omega}\|T_{\omega}x\|^{2}d\mu(\omega)\leq 2(1+\beta)[\underset{\omega \in \Omega}{\sup}( b_{\omega})]^{2}\int_{\Omega}\|\Gamma_{\omega}x\|^{2}d\mu(\omega).
	\end{equation*}
	This give:
	\begin{equation}\label{102}
	\frac{(1-2\alpha)[\underset{\omega \in \Omega}{\inf}( a_{\omega})]^{2}}{2(1+\beta)[\underset{\omega \in \Omega}{\sup}( b_{\omega})]^{2}}\int_{\Omega}\|T_{\omega}x\|^{2}d\mu(\omega)\leq \int_{\Omega}\|\Gamma_{\omega}x\|^{2}d\mu(\omega)
	\end{equation}
	From \eqref{101} and \eqref{102} we conclude,
	\begin{equation*}
	\frac{(1-2\alpha)[\underset{\omega \in \Omega}{\inf}( a_{\omega})]^{2}}{2(1+\beta)[\underset{\omega \in \Omega}{\sup}(
		 b_{\omega})]^{2}}\int_{\Omega}\|T_{\omega}x\|^{2}d\mu(\omega)\leq \int_{\Omega}\langle \Gamma_{\omega}x,\Gamma_{\omega}x\rangle_{\mathcal{A}} \mu(\omega)\leq \frac{ 2(1+\alpha)[\underset{\omega \in \Omega}{\sup}( a_{\omega})]^{2}}{(1-2\beta)[\underset{\omega \in \Omega}{\inf}( b_{\omega})]^{2}}\int_{\Omega}\|T_{\omega}x\|^{2}d\mu(\omega).
	\end{equation*}
	Hence, $\{\Gamma_{\omega}\}_{\omega\in\Omega}$ is an integral $K$-operator frame for $End_{\mathcal{A}}^{\ast}(\mathcal{H})$.
\end{proof}


\begin{thebibliography}{99}
	\bibitem{STAJP} S.T.Ali, J.P.Antoine and J.P.Gazeau, continuous frames in Hilbert spaces, Annals of physics 222 (1993), 1-37
	
	\bibitem{Ali} A.Alijani and M.A. Dehghan, $\ast$-frames in Hilbert $\mathcal{C}^{\ast}$-modules, {\it U.P.B. Sci. Bull., Ser. A}, {\bf 73}(4) (2011), 89-106.
	\bibitem{Ch} O. Christensen, An Introduction to Frames and Riesz bases, {\it Brikh\"{a}user}, 2016.
		\bibitem{Con} J.B.Conway, A Course In Operator Theory, {\it AMS}, {\bf 21}, 2000.
	\bibitem{Dav} F. R. Davidson, $\mathcal{C}^{\ast}$-algebra by example, {\it Fields Ins. Monog.} 1996.
	\bibitem{Duf} R. J. Duffin and A. C. Schaeffer, A class of nonharmonic fourier series, {\it Trans. Amer. Math. Soc.} {\bf 72} (1952),
	341-366.
	\bibitem{F4} M. Frank, D. R. Larson, Frames in Hilbert $\mathcal{C}^{\ast}$-modules and $\mathcal{C}^{\ast}$-algebras, J. Oper. Theory 48 (2002), 273-314.
	\bibitem{14} J. P. Gabardo and D. Han, \emph{Frames associated with measurable space}, Adv. Comp. Math. 18
	(2003), no. 3, 127-147.
	\bibitem{LG} L. Gavruta, Frames for operators, {\it Appl.Comput.Harmon.Anal.} {\bf 32} (2012), 139-144
	\bibitem{mjpaa} S.Kabbaj, H.Labrigui and A.Touri, Controlled continuous g-frames in Hilbert $C^{\ast}$-modules, Moroccan J. Of Pure and Appl. Anal. (MJPAA), Volume 6(2), 2020, Pages 184-197. 
	\bibitem{Kap} I. Kaplansky, Modules over operator algebras, {\it Amer. J. Math.} {\bf 75} (1953), 839-858.
	\bibitem{BA} A. Khosravi and B. Khosravi, Frames and bases in tensor products of Hilbert spaces and Hilbert $\mathcal{C}^{\ast}$-modules,
	{\it Proc. Indian Acad. Sci. Math. Sci.} {\bf 117} (2007), 1-12.
	\bibitem{moi} H.Labrigui, A. Touri and S.Kabbaj, Controlled Operators frames for $End_{\mathcal{A}}^{\ast}(\mathcal{H})$, Asian Journal Of Mathematics and Applications, Volume 2020, Article ID ama0554, 13 pages.
	\bibitem{moi1}Hatim Labrigui and Samir Kabbaj, Integral operator frames for $B(\mathcal{H})$, Journal of Interdisciplinary Mathematics, (2020) DOI:10.1080/09720502.2020.1781884. 
	\bibitem{Gav} A. Najati, M. M. Saem and P. Gavruta, Frames and operators in Hilbert $\mathcal{C}^{\ast}$-modules, {\it Oam}, {\bf 10}(1) (2016), 73-81.
	\bibitem{Pas} W. Paschke, Inner product modules over $B^{\ast}$-algebras, {\it Trans. Amer. Math. Soc.}, {\it 182}(1973), 443-468.
	\bibitem{ARAN} A.Rahmani, A.Najati, and Y.N.Deghan, Continuous frames in Hilbert spaces, methods of Functional Analysis and Topology Vol. 12(2), (2006), 170-182.
	\bibitem{MR1} M.Rahmani, On Some properties of c-frames, J.Math.Res.Appl, Vol. 37(4), (2017),466-476.
	\bibitem{MR2} M.Rahmani, Sum of c-frames, c-Riesz Bases and orthonormal mapping, U.P.B.Sci. Bull, Series A,Vol.77(3), (2015),3-14.
	\bibitem{r04} M.Rossafi, A.Bourouihiya, H.Labrigui and A.Touri, The duals of $\ast$-operator frames for $End_{\mathcal{A}}^{\ast}(\mathcal{H})$.
	\bibitem{r4} M. Rossafi, S. Kabbaj, \emph{K-operator Frame for $End_{\mathcal{A}}^{\ast}(\mathcal{H})$}, Asia Mathematika Volume {\bf 2}, Issue 2, (2018), 52-60.
	\bibitem{r5} M. Rossafi, S. Kabbaj, \emph{$\ast$-K-operator Frame for $End_{\mathcal{A}}^{\ast}(\mathcal{H})$},  Asian-European Journal of Mathematics, Vol. 13, No.03, 2050060 (2020).
	\bibitem{r7} M. Rossafi, S. Kabbaj, \emph{Frames and Operator Frames for $B(\mathcal{H})$}, Asia Mathematika Volume {\bf 2}, Issue 3, (2018), 19-23.
	\bibitem{r9} M. Rossafi, S. Kabbaj, \emph{Generalized Frames for $B(\mathcal{H, K})$}, accepted for publication in Iranian Journal of Mathematical Sciences and Informatics.
	\bibitem{r8} M. Rossafi, A. Touri, H. Labrigui and A. Akhlidj, \emph{Continuous $\ast$-K-G-Frame in Hilbert $C^{\ast}$-Modules}, Journal of Function Spaces, vol. 2019, Article ID 2426978, 5 pages, 2019. 
	\bibitem{33} K. Yosida, Functional Analysis, vol. 123 of Grundlehren der Mathematischen Wissenschaften, Springer,
	Berlin, Germany, 6th edition, 1980.
	\bibitem{Zha} L. C. Zhang, The factor decomposition theorem of bounded generalized inverse modules and their topological continuity, {\it J. Acta Math. Sin.}, {\bf 23} (2007), 1413-1418.
\end{thebibliography}
\end{document}